\documentclass[reqno]{amsart}
\usepackage{amssymb}
\usepackage{hyperref}
\usepackage{latexsym}
\usepackage{amsmath}

\usepackage[usenames,dvipsnames,svgnames,table]{xcolor}

\usepackage{amsfonts, amsmath, bbm, amssymb, amsthm}
\usepackage{enumerate}
\usepackage{todonotes}

\newcommand{\Z}{\mathbb{Z}}
\newcommand{\R}{\mathbb{R}}
\newcommand{\C}{\mathbb{C}}

\newcommand{\tr}{\boldsymbol{t}_\Sigma}

\newcommand{\dom}{\textup{dom }}

\newcommand{\sign}{\textup{sign}}

\newcommand{\abs}[1]{\lvert{#1}\rvert}

\newcommand{\norm}[1]{{\lVert{#1}\rVert}}

\theoremstyle{plain}
\newtheorem{thm}{Theorem}[section]
\newtheorem{lem}[thm]{Lemma}
\newtheorem{cor}[thm]{Corollary}
\newtheorem{prop}[thm]{Proposition}

\theoremstyle{definition}

\theoremstyle{remark}

\numberwithin{equation}{section}

\begin{document}

\title[Approximation of Dirac operators with confining $\delta$-shell potentials]
{Approximation of Dirac operators with confining electrostatic and Lorentz scalar $\boldsymbol{\delta}$-shell potentials}

%----------Author 1
\author[C.~Stelzer-Landauer]{Christian Stelzer-Landauer}

\address{%
	Institut für Angewandte Mathematik\\
	Technische Universität Graz\\
	Steyrergasse 30\\
	8010 Graz\\
	Austria}

\email{christian.stelzer09@gmail.com}

%----------Author 2
%\author{A Second Author}
%\address{The address of\br
%the second author\br
%sitting somewhere\br
%in the world}
%\email{dont@know.who.knows}


\date{January 1, 2025}
\dedicatory{Dedicated to Henk de Snoo on the occasion of his 80th birthday!}
\begin{abstract}
	In this paper we continue the investigations from \cite{BHS23,BHSL25} regarding the approximation of Dirac operators with $\delta$-shell potentials in the norm resolvent sense. In particular, we consider the approximation of Dirac operators with confining electrostatic and Lorentz scalar $\delta$-shell potentials, where the support of the $\delta$-shell potentials is impermeable to particles modelled by such Dirac operators.
\end{abstract}

%%% ----------------------------------------------------------------------
\maketitle
%%% ----------------------------------------------------------------------
%\tableofcontents

% ------------------------------------------------------------------------
\section{Introduction}\label{sec_int}

Dirac operators with $\delta$-shell potentials model the propagation of relativistic spin 1/2 particles in $\R^{\theta}$, $\theta \in \{2,3\}$, under the influence of an external field which is strongly localized around a hypersurface.
Mathematically, the self-adjointness and spectral properties of such operators were firstly considered in \cite{DES89} in 1989. After a period with little progress, the seminal paper \cite{AMV14}, which appeared in 2014, sparked renewed interest in this topic and resulted in numerous publications since then; see for instance \cite{BHSS22,CLMT21,R22a}. In the current paper we consider Dirac operators with $\delta$-shell potentials given by the (formal) expression
\begin{equation}\label{eq_H_V_tilde_formal}
	H_{\widetilde{V} \delta_\Sigma} = -i \sum_{j = 1}^{\theta} \alpha_j \partial_j + m \beta + \widetilde{V} \delta_\Sigma,
\end{equation}
where $\alpha_1,\dots,\alpha_\theta,\beta \in\C^{N \times N}$, $N = 2\lceil \theta/2\rceil$, are the Dirac matrices defined in Section~\ref{sec_not}~\eqref{it_Dirac_matrices},  $m \in \R$ represents the mass of the underlying particle and $\delta_\Sigma$ is a $\delta$-shell potential supported on the boundary $\Sigma$ of a $C^2$-smooth domain $\Omega_+$ which is specified in \eqref{it_def_Sigma} of Section~\ref{sec_not}. Moreover, $\widetilde{V} = \widetilde{\eta} I_N + \widetilde{\tau} \beta$, where $\widetilde{\eta} \in \R$ and $\widetilde{\tau} \in \R$ model electrostatic and Lorentz scalar interactions, respectively.  If $\widetilde{d} = \widetilde{\eta}^{\,2} - \widetilde{\tau}^{\,2} \neq 4$, then the initially formally defined operator $H_{\widetilde{V}\delta_\Sigma}$ can  be rigorously realized as a self-adjoint operator in $L^2(\R^\theta;\C^N)$ with $\dom H_{\widetilde{V}\delta_\Sigma} \subset H^{1}(\R^\theta \setminus \Sigma;\C^N)$ via transmission conditions on $\Sigma$; see, e.g.,   \cite{AMV15,BEHL19, BHOP20,OV16, R22a}. Here, the case 
\begin{equation}\label{eq_conf}
	\widetilde{d} = \widetilde{\eta}^{\,2} - \widetilde{\tau}^{\,2} = -4
\end{equation}
is particularly interesting since \eqref{eq_conf} implies that  $H_{\widetilde{V}\delta_\Sigma}$ splits into the orthogonal sum of two independent Dirac operators on the domains $\Omega_\pm$, where $\Omega_-:= \R^\theta \setminus \overline{\Omega_+}$. More precisely, $H_{\widetilde{V}\delta_\Sigma} = H_{\widetilde{V}}^+ \oplus H_{\widetilde{V}}^-$ with 
\begin{equation*}
	\begin{aligned}
		\dom H_{\widetilde{V}}^\pm &= \bigl\{ u_\pm \in H^1(\Omega_\pm;\C^N): 
		u_\pm|_\Sigma= \pm \tfrac{i}{2}(\alpha \cdot \nu)\widetilde{V}  u_\pm|_\Sigma  \bigr\}  \subset L^2(\Omega_\pm;\C^N),\\
		H_{\widetilde{V}}^\pm u_\pm &= -i(\alpha \cdot \nabla)u_\pm + m \beta u_\pm,
	\end{aligned}
\end{equation*}
where $\nu$ denotes the unit outward normal vector of $\Omega_+$ and $\alpha \cdot \nu := \sum_{j = 1}^\theta \alpha_j  \nu_j$;  see \cite[Theorem~5.5]{AMV15}, \cite[Section~5]{CLMT21} or \cite[Example~12]{R22a}. On a physical level, this means that $\Sigma$ becomes impermeable to a particle. Inspired by this observation we call $\widetilde{\eta}$ and $\widetilde{\tau}$ confining interaction strengths in this case. Furthermore, Dirac operators  on domains are of interest on their own as they are used in three dimensions to describe relativistic particles confined to a box, including the famous MIT bag model, and in two dimensions  they are applied in the description of graphene. For an overview of this topic we refer to the introductions of \cite{BHM20, BFSB17_1, BFSB17_2} and the references therein.

To justify the usage of $H_{\widetilde{V} \delta_\Sigma}$, one approximates Dirac operators with $\delta$-shell potentials by Dirac operators with strongly localized potentials in the strong or norm resolvent sense. Strong and norm resolvent convergence are well-suited for these approximation problems as they  relate the spectra of the approximating operators and the limit operators. In order to define  strongly localized potentials, we first introduce the  mapping
\begin{equation}\label{eq_iota}
	\iota: \Sigma \times \R \to \R^\theta, \quad \iota(x_\Sigma,t):= x_\Sigma + t \nu(x_\Sigma), \quad (x_\Sigma,t) \in \Sigma \times \R,
\end{equation}
and define  for $\varepsilon >0$ the set $\Omega_\varepsilon := \iota(\Sigma \times (-\varepsilon,\varepsilon))$, which is the so-called \textit{tubular neighbourhood} of $\Sigma$. According to  \cite[Proposition~2.4]{BHS23}, we can fix  an 
$\varepsilon_1>0$  such that   $\iota \upharpoonright \Sigma \times (-\varepsilon_1,\varepsilon_1)$ is injective. Equipped with these notations we define based on 
\begin{gather}
	V = \eta I_N + \tau \beta,\, \quad \eta,\tau \in \R, \label{eq_V}\\
	q \in L^{\infty}((-1,1);[0,\infty))  \text{ with }   \int_{-1}^{1} q(s)  \, ds = 1, \label{eq_q}
\end{gather}
for $\varepsilon \in (0,\varepsilon_1)$ the strongly localized potentials
\begin{equation}\label{eq_V_eps}
	V_\varepsilon(x) := 
	\begin{cases}
		\frac{1}{\varepsilon}Vq(t/\varepsilon),&  x = \iota(x_\Sigma,t) \in \Omega_\varepsilon,\\
		0,&  x \notin \Omega_\varepsilon.
	\end{cases}
\end{equation} 
Moreover, for $m \in \R$ and $\varepsilon \in (0, \varepsilon_1)$ we introduce the self-adjoint operators
\begin{equation*}
	H_{V_\varepsilon}  := -i(\alpha \cdot \nabla)  + m \beta  + V_\varepsilon  , \quad  \dom H_{V_\varepsilon} :=  H^1(\R^\theta;\C^N);
\end{equation*}
cf. \cite[eq. (1.8)]{BHS23}.
It is known that $H_{V_\varepsilon}$ converges for $\varepsilon \to 0$ in the norm  or  strong resolvent sense (depending on the parameter $d= \eta^2 - \tau^2$) to $H_{\widetilde{V}\delta_\Sigma}$  with  $\widetilde{V}= \textup{tanc}(\sqrt{d}/2) V$, where $\textup{tanc}$ is the function defined in Section~\ref{sec_not}~\eqref{it_tanc}; see \cite{BHS23, BHSL25, BHT23,CLMT21,MP18,Z23}.  Using the rescaling formula yields $\widetilde{V} = \widetilde{\eta} I_N + \widetilde{\tau} \beta$ with $(\widetilde{\eta},\widetilde{\tau}) = \textup{tanc}(\sqrt{d}/2)(\eta,\tau)$ and hence
\begin{equation}\label{eq_rescaling_d}
	\widetilde{d} = \widetilde{\eta}^{\,2} - \widetilde{\tau}^{\,2} = d \textup{tanc}(\sqrt{d}/2)^2 = 4 \tan(\sqrt{d}/2)^2.
\end{equation} 
Note that $\widetilde{d}\geq0$ for $d\geq0$ and $\widetilde{d} =-4\tanh(\sqrt{\abs{d}}/2)$ for $d <0$. In particular, $\widetilde{d}$ is  bigger than $-4$. This means that $H_{\widetilde{V}\delta_\Sigma}$ with $\widetilde{d} = -4$, i.e. Dirac operators with confining interaction strengths, cannot be approximated with  this approach. This raises the question whether there is also a way to approximate such Dirac operators by Dirac operators with strongly localized potentials.
Inspired by the rescaling of $d$ from \eqref{eq_rescaling_d}, one would have to choose $V= \eta I_N + \tau \beta$ with $d = \eta^2-\tau^2 = -\infty$ as the interaction matrix in the approximating operators to obtain a Dirac operator with a $\delta$-shell potential and $ \widetilde{d} = \widetilde{\eta}^{\,2}-\widetilde{\tau}^{\,2}= -4$ in the limit. We rigorously  realize this idea by choosing $\varepsilon$-dependent interaction strengths $\eta_\varepsilon$ and $\tau_\varepsilon$ such that $\eta_\varepsilon^2 - \tau_\varepsilon^2 \overset{\varepsilon \to 0}{\longrightarrow}  -\infty$. The same approach was used in \cite[Chapter~3]{Ru21} when dealing with the one-dimensional counterpart. See also \cite{Z24} for related considerations.

Next, let us describe this approach in more detail. First, we  fix a continuous scaling function $f: (0,\varepsilon_1) \to (0,\infty)$, which satisfies the  condition 
\begin{equation}\label{eq_conf_f}
	\lim_{\varepsilon \to 0} f(\varepsilon) = \infty \quad \textup{and} \quad \lim_{\varepsilon \to 0} f(\varepsilon)^{3/2}\varepsilon^\gamma = 0
\end{equation} 
for a $\gamma \in (0,1/2)$. We also introduce the $\varepsilon$-dependent interaction strengths
\begin{equation*}
	(\eta_\varepsilon,\tau_\varepsilon) =f(\varepsilon)(\eta,\tau)\quad \textup{and} \quad  d_\varepsilon = \eta_\varepsilon^2 -\tau_\varepsilon^2 = f(\varepsilon)^2d
\end{equation*}
and  assume from now on
\begin{equation}\label{eq_conf_V}
	d = \eta^2 - \tau^2  <0,
\end{equation}
which guarantees
\begin{equation*}
	d_\varepsilon = \eta_\varepsilon^2 -\tau_\varepsilon^2 = f(\varepsilon)^2 d \overset{\varepsilon \to 0}{\longrightarrow}  - \infty.
\end{equation*}
Dirac operators with strongly localized potentials and the $\varepsilon$-dependent interaction matrix $\eta_\varepsilon I_N + \tau_\varepsilon \beta = f(\varepsilon)V$ are given by
\begin{equation*}
	H_{f(\varepsilon)V_\varepsilon} u := -i(\alpha \cdot \nabla) + m \beta  + f(\varepsilon)V_\varepsilon , \quad  \dom H_{f(\varepsilon) V_\varepsilon} :=  H^1(\R^\theta;\C^N),
\end{equation*}
where $V_\varepsilon $ is defined by \eqref{eq_V_eps}.
Note that for a fixed $\varepsilon' \in (0,\varepsilon_1)$ the operator $H_{f(\varepsilon')V_\varepsilon}$ converges by \cite[Theorem~2.1]{BHSL25} for $\varepsilon \to 0$ in the norm resolvent sense to $H_{\widetilde{V}_{\varepsilon'} \delta_\Sigma}$, where  
\begin{equation}\label{eq_V_eps_tilde}
	\widetilde{V}_{\varepsilon'} := \textup{tanc}(\sqrt{d_{\varepsilon'}}/2)f(\varepsilon')V.
\end{equation}
Relying on the methods developed in \cite{BHS23,BHSL25}, we prove in the present paper that the  norm resolvent limit of $H_{f(\varepsilon)V_\varepsilon}$ is the operator $H_{\widetilde{V} \delta_\Sigma}$ with 
\begin{equation*}
	\begin{aligned}
		\widetilde{V} = \lim_{\varepsilon \to 0} 	\widetilde{V}_{\varepsilon} &= \lim_{\varepsilon \to 0} \textup{tanc}(\sqrt{d_\varepsilon}/2) f(\varepsilon) V \\
		&= \lim_{\varepsilon \to 0} \frac{2\tan(f(\varepsilon)\sqrt{d}/2 )}{ f(\varepsilon)\sqrt{d}}f(\varepsilon) V \\
		&= \lim_{\varepsilon \to 0} \frac{2\tan(f(\varepsilon)\sqrt{d}/2 )}{ \sqrt{d}} V\\
		&= \lim_{\varepsilon \to 0} \frac{2\tanh(f(\varepsilon)\sqrt{\abs{d}}/2)}{ \sqrt{\abs{d}}} V= \frac{2}{\sqrt{|d|}} V.
	\end{aligned}
\end{equation*}
Consequently, $\widetilde{V} = \widetilde{\eta} I_N + \widetilde{\tau}  \beta $ with $(\widetilde{\eta},\widetilde{\tau}) = (2/\sqrt{|d|})(\eta,\tau)$ and, in turn, we have  $\widetilde{d} = \widetilde{\eta}^{\,2}-  \widetilde{\tau}^{\,2} = -4$, i.e.  $\widetilde{\eta}$ and $\widetilde{\tau}$ are confining interaction strengths.

To  prove the norm resolvent convergence, we find for $z \in \C \setminus \R$ the estimates 
\begin{equation}\label{eq_res_dif_1}
	\lVert (H_{\widetilde{V}\delta_\Sigma}-z)^{-1} - (H_{\widetilde{V}_\varepsilon \delta_\Sigma}-z)^{-1} \rVert_{L^2(\R^\theta;\C^N)\to L^2(\R^\theta;\C^N)} \leq C e^{-f(\varepsilon) \sqrt{\abs{d}}}
\end{equation}
and 
\begin{equation}\label{eq_res_dif_2}
	\lVert(H_{\widetilde{V}_\varepsilon \delta_\Sigma}-z)^{-1}-(H_{f(\varepsilon)V_\varepsilon}-z)^{-1}\rVert_{L^2(\R^\theta;\C^N)\to L^2(\R^\theta;\C^N)} \leq C f(\varepsilon)^{3/2}\varepsilon^{\gamma}
\end{equation}
in Section~\ref{sec_res_dif_1} (Proposition~\ref{prop_dif_1}) and Section~\ref{sec_res_dif_2} (Proposition~\ref{prop_dif_2}), respectively. Using these two estimates and the triangle inequality yields the main result of this paper, which reads as follows.
\begin{thm}\label{theo_conf}
	Let  $q$ be as in \eqref{eq_q}, $V = \eta I_N + \tau \beta $, $\eta,\tau \in\R$,  satisfy  \eqref{eq_conf_V}, $V_\varepsilon$ be defined by \eqref{eq_V_eps}, $f$ be as in \eqref{eq_conf_f} \textup{(}with $\gamma \in (0,1/2)$\textup{)} and  $z \in \C\setminus\R$. Moreover, set $\widetilde{V} = (2/\sqrt{|d|})V$, where $d = \eta^2 - \tau^2$. Then, there exists a  $\delta > 0$ and a $C>0$ such that
	\begin{equation*}
		\begin{aligned}
			\|(H_{\widetilde{V} \delta_\Sigma} -z)^{-1} - (H_{f(\varepsilon)V_\varepsilon} -z)^{-1}&\|_{L^2(\R^{\theta};\C^N) \to L^2(\R^{\theta};\C^N)} \\
			&\qquad \leq C \bigl( e^{-f(\varepsilon) \sqrt{\abs{d}}} + f(\varepsilon)^{3/2} \varepsilon^{\gamma} \bigr)
		\end{aligned}
	\end{equation*}
	for all $ \varepsilon \in  (0,\delta)$. In particular, $H_{f(\varepsilon) V_\varepsilon}$ converges to $H_{\widetilde{V}\delta_\Sigma}$ in the norm resolvent sense as $\varepsilon \to 0$.
\end{thm}
Note that if  $V = \eta I_N + \tau \beta$, $\eta,\tau \in \R$, such that $d = \eta^2 - \tau^2 = -4$, then
\begin{equation*}
	\widetilde{V} =  \frac{2}{\sqrt{\abs{d}}} V = V.
\end{equation*}
Hence, we obtain the following corollary.
\begin{cor}\label{cor_conf}
	Let  $q$ be as in \eqref{eq_q}, $V = \eta I_N +\tau \beta$,  $\eta,\tau \in \R$, fulfil the equation $d =\eta^2 -\tau^2 =-4$, $V_\varepsilon$ be defined by \eqref{eq_V_eps}, $f$ be as in \eqref{eq_conf_f} \textup{(}with $\gamma \in (0,1/2)$\textup{)} and  $z \in \C\setminus\R$. Then, there exists a $\delta >0 $ and  a $C>0$ such that
	\begin{equation*}
		\begin{aligned}
			\|(H_{f(\varepsilon)V_\varepsilon} -z)^{-1} - (H_{V \delta_\Sigma} -z)^{-1}&\|_{L^2(\R^{\theta};\C^N) \to L^2(\R^{\theta};\C^N)} \\
			&\qquad \leq C \bigl( e^{-f(\varepsilon) \sqrt{\abs{d}}} + f(\varepsilon)^{3/2} \varepsilon^{\gamma} \bigr)
		\end{aligned}
	\end{equation*}
	for all $ \varepsilon \in  (0,\delta)$. In particular, $H_{f(\varepsilon) V_\varepsilon}$ converges to $H_{V \delta_\Sigma}$ in the norm resolvent sense as $\varepsilon \to 0$.
\end{cor}
This corollary shows that every Dirac operator with given constant confining Lorentz scalar and electrostatic interaction strengths is the norm resolvent limit of Dirac operators with strongly localized potentials.

Finally, let us mention that the present paper is based on Section~7.1 of the dissertation \cite{SL24}, where for position dependent interaction strengths norm resolvent convergence with a weaker convergence rate was proven.

\subsection{Notations and Assumptions}\label{sec_not}
In this paper we use the following notations and assumptions:
\begin{enumerate}[(i)]
	\item By $\theta \in \{ 2, 3 \}$ we denote the space dimension and we set $N=2$ for $\theta=2$ and $N=4$ for $\theta = 3$. 
	\item \label{it_Dirac_matrices} The Pauli spin matrices are given by
	\begin{equation*}
		\sigma_1 = \begin{pmatrix} 0 & 1\\ 1 & 0 \end{pmatrix}, \quad\sigma_2 = \begin{pmatrix} 0 & -i\\ i & 0 \end{pmatrix} \quad \text{and} \quad \sigma_3 =\begin{pmatrix} 1 & 0\\ 0 & -1 \end{pmatrix}.
	\end{equation*} 
	With their help the Dirac matrices $\alpha_1, \dots, \alpha_\theta, \beta \in \mathbb{C}^{N \times N}$ are defined for $\theta=2$ by
	\begin{equation*}
		\alpha_1 := \sigma_1, \quad \alpha_2 := \sigma_2 \quad \text{and} \quad \beta := \sigma_3,
	\end{equation*}
	and  for $\theta = 3$ by
	\begin{equation*}
		\alpha_j := \begin{pmatrix} 0 &
			\sigma_j \\ \sigma_j & 0
		\end{pmatrix}\text{ for } j=1,2,3  \quad \text{and} \quad \beta := \begin{pmatrix}
			I_2 & 0 \\ 0 & -I_2
		\end{pmatrix},
	\end{equation*}
	where $I_2$ is the $2 \times 2$-identity matrix. The Dirac matrices satisfy for $j,k \in \{1,\dots,\theta\}$ the rules
	\begin{equation*}
		\alpha_j \alpha_k + \alpha_k \alpha_j = 2 I_N \delta_{jk} \quad \text {and} \quad \alpha_j \beta + \beta \alpha_j=0,
	\end{equation*}
	where $\delta_{jk}$ denotes the Kronecker delta and $I_N$ is the $N \times N$ identity matrix.
	We will also often make use of the notations
	\begin{equation*}
		\alpha \cdot \nabla := \sum_{j = 1}^\theta \alpha_j \partial_j  \text{ and }  \alpha \cdot x := \sum_{j = 1}^\theta \alpha_j x_j, \,\, x = (x_1, \dots, x_\theta) \in \R^\theta,
	\end{equation*}
	and the identity $(\alpha \cdot x)^2 = I_N \abs{x}^2$, $x \in \R^{\theta}$.
	\item \label{it_def_Sigma} We assume that $\Sigma$ is  the boundary of an open set $\Omega_+ \subset \R^{\theta}$ which satisfies the following: There exist open sets  $W_1, \dots ,W_p \subset \mathbb{R}^\theta$, mappings $\zeta_1,\dots,\zeta_p \in C^{2}_b(\R^{\theta-1}; \mathbb{R})$, rotation matrices $\kappa_1, \dots ,\kappa_p \in \R^{\theta \times \theta}$ and a number $\varepsilon_0 >0$ such that
	\begin{itemize}
		\item[(i)] $\Sigma \subset \bigcup_{l=1}^p W_l$;
		\item[(ii)] if $x \in \partial \Omega_+ = \Sigma$, then there exists an $l \in \{1, \dots ,p\}$ such that $B(x,\varepsilon_0) \subset  W_l$;
		\item[(iii)] $W_l \cap \Omega_+
		= W_l \cap \Omega_l$, where 
		\begin{equation*}
			\Omega_l := \biggl\{ \kappa_l \begin{pmatrix} x' \\ x_\theta
			\end{pmatrix}: x_\theta < \zeta_l(x'), \, \begin{pmatrix} x' \\ x_\theta
			\end{pmatrix} \in \R^\theta \biggr\}
		\end{equation*}
		for $ l \in \{1,\dots,p\}$.
	\end{itemize}
	Furthermore, we  denote the unit normal vector field at $\Sigma$ that is pointing outwards of $\Omega_+$ by $\nu$ and for $u: \mathbb{R}^\theta \rightarrow \mathbb{C}^N$ we define $u_\pm := u \upharpoonright \Omega_\pm$. Finally, let us mention that $W$ denotes the Weingarten map corresponding to $\Sigma$;  see \cite[Definition~2.3 and Proposition~2.4]{BHS23} for details. 
	
	\item Let $\mathcal{H}$ and $\mathcal{G}$ be Hilbert spaces and  $A$ be a linear operator from $\mathcal{H}$ to $\mathcal{G}$. The domain of $A$ is denoted by $\dom A$. If $A$ is bounded and everywhere defined, then we write $\| A \|_{\mathcal{H} \rightarrow \mathcal{G}}$ for its operator norm. If $\mathcal{H} = \mathcal{G}$ and $A$ is a closed operator, then the resolvent set and the spectrum  of $A$ are denoted by $\rho(A)$ and  $\sigma(A)$, respectively.
	
	\item\label{it_Sobolev}
	If $U \subset \R^\theta$ is open, then  $H^r(U;\C^N)$ denotes the $L^2(U;\C^N)$-based Sobolev space of order $r$; cf. \cite[Chaper~3]{M00}. Moreover, we also use the Sobolev trace space $H^r(\Sigma;\C^N)$, $r \in [-2,2]$, which is  defined via partition of unity as in  \cite[Section~2.1]{BHS23}.
	The well-defined and bounded Dirichlet trace operators are denoted by
	\begin{equation*}
		\begin{aligned}
			\tr^{\pm}&:  H^r(\Omega_\pm;\C^N) \to H^{r-1/2}(\Sigma;\C^N),\\
			\tr &: H^{r}(\R^{\theta};\C^N) \to H^{r-1/2}(\Sigma;\C^N),
		\end{aligned}
	\end{equation*}
	$r \in (1/2,5/2)$; cf. \cite[Theorem~2]{M87}. Furthermore, for 
	\begin{equation*}
		u \in H^r(\R^{\theta} \setminus \Sigma;\C^N) = H^r(\Omega_+;\C^N) \oplus H^r(\Omega_-;\C^N)
	\end{equation*} 
	we use the shortened notation $\tr^\pm u$ instead of $\tr^\pm u_\pm$. 
	
	\item\label{it_Bochner} The  $L^2((-1,1))$-based Bochner-Lebesgue space of $H^r(\Sigma;\C^N)$-valued
	functions $L^2((-1,1);H^r(\Sigma;\C^N))$, cf. \cite[Section~2.2]{BHS23}, is denoted by the symbol $\mathcal{B}^r(\Sigma)$. We also write
	$\norm{\cdot}_r$ for the norm in $\mathcal{B}^r(\Sigma)$. In a similar way, we define
	\begin{equation*}\begin{aligned} 
			\norm{\cdot}_{r \to r'} &:=\norm{\cdot}_{\mathcal{B}^r(\Sigma) \to \mathcal{B}^{r'}(\Sigma)},\\
			\norm{\cdot}_{r \to \mathcal{H}} &:= \norm{\cdot}_{\mathcal{B}^r(\Sigma) \to \mathcal{H}},\\
			\norm{\cdot}_{\mathcal{H} \to r'} &:= \norm{\cdot}_{\mathcal{H} \to \mathcal{B}^{r'}(\Sigma) }.
		\end{aligned}
	\end{equation*}
	We   use the following convenient identification: Let  $\mathcal{A}$ be a bounded operator in $H^r(\Sigma;\C^N)$, $r \in [-2,2]$,
	and $\mathcal{Q} \in L^\infty((-1,1))$. Then, we identify 
	\begin{equation*}
		\mathcal{M}_{\mathcal{Q}}: \mathcal{B}^r(\Sigma)\to \mathcal{B}^r(\Sigma),\quad 
		(\mathcal{M}_{\mathcal{Q}}f)(t) := \mathcal{Q}(t) f(t), 
	\end{equation*}
	and 
	\begin{equation*}
		\mathcal{M}_{\mathcal{A}} : \mathcal{B}^r(\Sigma)\to \mathcal{B}^r(\Sigma) ,\quad
		(\mathcal{M}_{\mathcal{A}}f)(t) :=  \mathcal{A}(f(t)), 
	\end{equation*}
	with $\mathcal{Q}$ and $\mathcal{A}$, respectively. Note  that   $\norm{\mathcal{M}_{\mathcal{Q}}}_{r \to r}$ and $\norm{\mathcal{M}_{\mathcal{A}}}_{r \to r}$ are equal  to $\norm{\mathcal{Q}}_{L^\infty((-1,1))}$ and $\norm{\mathcal{A}}_{H^r(\Sigma;\C^N)\to H^r(\Sigma;\C^N)}$, respectively. 
	
	\item The letter $C>0$ always denotes a generic constant which may change in-between lines.

	\item The branch of the square root is fixed by $\text{Im}\, \sqrt{w} > 0$ for $w \in \C \setminus [0, \infty)$.
	
	\item \label{it_tanc} For $w\in \C \setminus \{k\pi + \pi/2 : k \in \Z \}$ we define the function 	
	\begin{equation*}
		\textup{tanc}(w) :=
		\begin{cases}
			\tfrac{\tan(w)}{w}, &  w \in \C \setminus \big(\{0\} \cup \{ k\pi + \pi/2 : k \in \Z \}\big),\\
			1, & w = 0.
		\end{cases}
	\end{equation*}
	For $ x \in \R$ the equation $\textup{tanc}(i x) = \tanh(x)/x$ is valid.
	
\end{enumerate}

\section{A bound for (\ref{eq_res_dif_1})}\label{sec_res_dif_1}
Before we estimate \eqref{eq_res_dif_1}, we provide a resolvent formula for Dirac operators with $\delta$-shell potentials in terms of the resolvent of the free Dirac operator and potential and boundary integral operators associated to the Dirac operator. The free Dirac operator
\begin{equation*}
	\begin{aligned}
		H  &:= - i (\alpha \cdot \nabla)  + m \beta , \qquad \dom H := H^1(\R^\theta;\C^N),
	\end{aligned}
\end{equation*}
is self-adjoint in $L^2(\R^\theta;\C^N)$ and $\sigma(H) = (-\infty,-\lvert m \rvert]\cup[\lvert m \rvert,\infty)$;  see for instance \cite[Section~2]{BHT23} for $\theta=2$ and \cite[Theorem 1.1]{T92} for $\theta =3$. Moreover, for $z \in \rho(H) = \C \setminus \sigma(H)$, 
\begin{equation*}
	\begin{aligned}
		R_z u(x) &:= (H-z)^{-1} u(x) \\
		&= \int_{\mathbb{R}^\theta} G_z(x-y) u(y) \,d y, \quad u \in L^2(\mathbb{R}^\theta; \mathbb{C}^N), ~x \in \mathbb{R}^\theta,
	\end{aligned}
\end{equation*}
where  $G_z$ is given for $\theta=2$ and $x \in \mathbb{R}^2 \setminus \{ 0 \}$ by
\begin{equation*}
	\begin{aligned}
		G_z(x) &= \frac{\sqrt{ z^2-m^2}}{2\pi} K_1\big(-i \sqrt{ z^2-m^2}\abs{x}\big)\frac{\alpha \cdot x}{\abs{x}} 
		\\
		&\hspace{100 pt}+\frac{1}{2\pi} K_0\big(-i \sqrt{ z^2-m^2}\abs{x}\big)\left(m\beta +  z I_2\right)
	\end{aligned}
\end{equation*}	
and for $\theta=3$ and $x \in \mathbb{R}^3 \setminus \{ 0 \}$ by
\begin{equation*}
	G_ z(x) = \left(  z I_4 + m \beta + i\left( 1 - i \sqrt{ z^2 -m^2}\abs{x} \right) \frac{ \alpha \cdot x }{\abs{x}^2}\right)\frac{e^{i\sqrt{ z^2 -m^2} \abs{x}}}{4 \pi \abs{x}};
\end{equation*}
cf., e.g., \cite[eq. (3.2)]{BHOP20}, \cite[eq. (1.19)]{BHSS22} or \cite[eq. (1.263)]{T92}. Here, $K_0$ and $K_1$ denote the modified Bessel functions of the second kind of order zero and one, respectively.
Next, let us introduce for $z \in \rho(H)$ the potential integral operator 
\begin{equation}\label{eq_Phi_z}
	\begin{aligned}
		\Phi_z: L^2(\Sigma;\C^N) &\to L^2(\R^\theta;\C^N),\\
		\Phi_z \psi (x) &= \int_{\Sigma} G_z(x-y_\Sigma)\psi(y_\Sigma) \,d\sigma(y_\Sigma).
	\end{aligned}
\end{equation}
By \cite[Lemma~2.1]{AMV14} this operator is well-defined and bounded. Additionally, $\Phi_z$ also acts  as a bounded operator from $H^r(\Sigma;\C^N)$ to  $H^{r+1/2}(\R^\theta\setminus \Sigma;\C^N)$ for any $r \in [0,1/2]$ and $\Phi_z^*$ is bounded as an operator from $L^2(\R^\theta;\C^N)$ to $H^{1/2}(\Sigma;\C^N)$; see \cite[Proposition~2.8]{BHS23}. The boundary integral operator $\mathcal{C}_z:L^2(\Sigma;\C^N) \to L^2(\Sigma;\C^N)$ is defined as the unique bounded extension of
\begin{equation}\label{eq_C_z}
	H^{1/2}(\Sigma;\C^N) \ni \varphi \mapsto \frac{1}{2}(\tr^+  + \tr^-)\Phi_z \varphi \in H^{1/2}(\Sigma;\C^N).
\end{equation}
According to \cite[Proposition 2.9]{BHS23} $\mathcal{C}_z$ is well-defined and also acts as a bounded operator in $H^r(\Sigma;\C^N)$ for $r \in [-1/2,1/2]$. Finally, let us mention that the operator $H_{\widetilde{V}\delta_\Sigma}$, which  is formally given by \eqref{eq_H_V_tilde_formal}, can be rigorously realized via transmission conditions as an unbounded operator in $L^2(\R^\theta;\C^N)$ by
\begin{equation*}
	\begin{split}
		\dom H_{\widetilde{V}\delta_\Sigma}&:= \biggl\{ u \in  H^1 (\R^\theta \setminus \Sigma;\C^N) =H^1 (\Omega_+;\C^N) \oplus H^1(\Omega_-; \mathbb{C}^N): \\
		&  \hspace{90 pt} i(\alpha \cdot \nu )(\tr^+   - \tr^- )u +  \frac{\widetilde{V}}{2}(\tr^+   + \tr^- )u =0 \biggr\}, \\
		H_{\widetilde{V}\delta_\Sigma}u & := (-i(\alpha \cdot \nabla) + m \beta) u_+ \oplus (-i(\alpha \cdot \nabla ) + m \beta) u_-;
	\end{split}
\end{equation*}
see \cite[eqs.~(1.1)--(1.3)]{BEHL19} or \cite[Section~4.1]{OV16}.
Having all the necessary notations at hand, we provide a self-adjointness criteria and a resolvent formula for $H_{\widetilde{V}\delta_\Sigma}$.

\begin{prop}\label{prop_H_V_hat}
	Let $\widetilde{V} = \widetilde{\eta}I_N + \widetilde{\tau} \beta$, $\widetilde{\eta} , \widetilde{\tau} \in \R$,  satisfy $\widetilde{d} = \widetilde{\eta}^{\,2}-\widetilde{\tau}^{\,2} \neq 4$.
	Then, $H_{\widetilde{V} \delta_\Sigma}$ is self-adjoint in $L^2(\R^\theta;\C^N)$. Moreover, for $z \in \C\setminus \R$ the operator $I+ \mathcal{C}_z\widetilde{V}$ is boundedly invertible in $H^r(\Sigma)$, $r \in [0,1/2]$, and the resolvent formula   
	\begin{equation*}
		(H_{\widetilde{V} \delta_\Sigma} - z)^{-1} = R_z - \Phi_z \widetilde{V}(I+\mathcal{C}_z\widetilde{V})^{-1} \Phi_{\overline{z}}^*
	\end{equation*}
	applies.
\end{prop}
\begin{proof}
	The self-adjointness of $H_{\widetilde{V} \delta_\Sigma}$ and the bounded invertibility of the operator $I+ \mathcal{C}_z\widetilde{V}$ follow from \cite[Section~6]{R22a} and \cite[the text above Section~4.1 and Lemma~4.8]{BHSL25}, respectively. It remains to prove the resolvent formula. We start by mentioning that according to \cite[Proposition~2.8~(ii) and (iii), and Proposition~2.9~(ii)]{BHS23}  the following identities are valid for $\varphi \in H^{1/2}(\Sigma;\C^N)$ and $v \in L^2(\R^\theta;\C^N)$
	\begin{equation}\label{eq_C_z_Phi_z_relations}
		\begin{aligned}
			\mathcal{C}_z \varphi &= \pm \frac{i}{2}(\alpha \cdot \nu) \varphi + \tr^\pm\Phi_z \varphi,\\
			0 &= (-i(\alpha \cdot \nabla) + m \beta - z I_N)(\Phi_z \varphi)_\pm,\\
			\Phi_{\overline{z}}^*v &= \tr R_z v. 
		\end{aligned}
	\end{equation} 
	Now, we fix $v \in L^2(\R^\theta;\C^N)$ and set $u := R_z v - \Phi_z \widetilde{V}(I+\mathcal{C}_z\widetilde{V})^{-1} \Phi_{\overline{z}}^*v$. Then, the mapping properties of $\Phi_z$ and $\Phi_{\overline{z}}^*$ discussed below \eqref{eq_Phi_z} and \eqref{eq_C_z},  and  $R_zv \in \dom H = H^1(\R^\theta;\C^N) \subset H^1(\R^\theta \setminus \Sigma;\C^N) $ show $u \in H^1(\R^\theta \setminus \Sigma;\C^N)$. Moreover, using $R_z v \in H^1(\R^\theta;\C^N)$, \eqref{eq_C_z_Phi_z_relations} (with $\varphi = \Phi_z \widetilde{V}(I+\mathcal{C}_z\widetilde{V})^{-1} \Phi_{\overline{z}}^*v$) and $(\alpha \cdot \nu)^2 = I_N$ gives us 
	\begin{equation*}
		\begin{aligned}
			\frac{\widetilde{V}}{2}(\tr^+  + \tr^-)u &= \widetilde{V}\tr R_zv - \frac{\widetilde{V}}{2}(\tr^+ + \tr^-) \Phi_z \widetilde{V}(I+\mathcal{C}_z\widetilde{V})^{-1} \Phi_{\overline{z}}^*v \\
			&=	\widetilde{V}\Phi_{\overline{z}}^* v - \widetilde{V}\mathcal{C}_z  \widetilde{V}(I+\mathcal{C}_z\widetilde{V})^{-1} \Phi_{\overline{z}}^*v\\
			&= \widetilde{V}(I+\mathcal{C}_z\widetilde{V})(I+\mathcal{C}_z\widetilde{V})^{-1}\Phi_{\overline{z}}^* v - \widetilde{V} \mathcal{C}_z  \widetilde{V}(I+\mathcal{C}_z\widetilde{V})^{-1} \Phi_{\overline{z}}^*v\\
			&= \widetilde{V}(I+\mathcal{C}_z\widetilde{V})^{-1} \Phi_{\overline{z}}^*v\\
			&= -i(\alpha \cdot \nu)i(\alpha \cdot \nu)\widetilde{V}(I+\mathcal{C}_z\widetilde{V})^{-1} \Phi_{\overline{z}}^*v\\
			&= i (\alpha \cdot \nu) (\tr^+ - \tr^-) \Phi_z\widetilde{V}(I+\mathcal{C}_z\widetilde{V})^{-1} \Phi_{\overline{z}}^*v\\
			&=  -i (\alpha \cdot \nu) (\tr^+ - \tr^-) \bigl(R_z v- \Phi_z\widetilde{V}(I+\mathcal{C}_z\widetilde{V})^{-1} \Phi_{\overline{z}}^*v\bigr)\\
			&= -i (\alpha \cdot \nu) (\tr^+ - \tr^-) u ;
		\end{aligned}
	\end{equation*}
	i.e. $u \in \dom H_{\widetilde{V} \delta_\Sigma}$. Furthermore, the second equation in \eqref{eq_C_z_Phi_z_relations} and the identity $(H-z)R_z u = u$ yield
	\begin{equation*}
		\begin{aligned}
			((H_{\widetilde{V} \delta_\Sigma}-z) u)_\pm &= (-i (\alpha \cdot \nabla) + m \beta - zI_N)u_\pm \\
			&= (-i (\alpha \cdot \nabla) + m \beta - zI_N) (R_z v)_\pm \\
			&\quad - (-i (\alpha \cdot \nabla) + m \beta- zI_N) (\Phi_z \widetilde{V}(I+\mathcal{C}_z\widetilde{V})^{-1} \Phi_{\overline{z}}^*v)_\pm\\
			& =  ((H-z)R_z v)_\pm = v_\pm,
		\end{aligned}
	\end{equation*} 
	which completes the proof of the resolvent formula.
\end{proof}

Next, we use Proposition~\ref{prop_H_V_hat} to estimate \eqref{eq_res_dif_1}.

\begin{prop}\label{prop_dif_1}
	Let $q$  be as in \eqref{eq_q}, $V = \eta I_N + \tau \beta$, $\eta,\tau \in \R$, satisfy \eqref{eq_conf_V},  $V_\varepsilon$ be defined by \eqref{eq_V_eps}, $f$ be  as in \eqref{eq_conf_f}  and  $z \in \C\setminus\R$. Moreover, let  ${\widetilde{V} = (2/\sqrt{|d|}})V$ and $\widetilde{V}_\varepsilon =  \textup{tanc}(\sqrt{d_\varepsilon}/2)f(\varepsilon) V,$ where $d = \eta^2 - \tau^2$ and $d_\varepsilon = f(\varepsilon)^2 d$. Then, there exists a $\delta_1 > 0$ and a $C>0$ such that
	\begin{equation*}
		\norm{(H_{\widetilde{V}\delta_\Sigma} -z)^{-1} - (H_{\widetilde{V}_\varepsilon \delta_\Sigma} -z)^{-1}}_{L^2(\R^{\theta};\C^N) \to L^2(\R^{\theta};\C^N)} \leq C  e^{-f(\varepsilon)\sqrt{\abs{d}}}  
	\end{equation*}
	for all $\varepsilon \in (0,\delta_1)$.
\end{prop}
\begin{proof}
	The interaction strengths corresponding to the interaction matrix $\widetilde{V}_\varepsilon$ are given by  $(\widetilde{\eta}_\varepsilon, \widetilde{\tau}_\varepsilon)= \textup{tanc}(\sqrt{d_\varepsilon}/2) f(\varepsilon) (\eta,\tau)$. Using  $d <0$, see \eqref{eq_conf_V}, yields
	\begin{equation*}
		\begin{aligned}
			\widetilde{d}_\varepsilon &= \widetilde{\eta}_\varepsilon^{\, 2} - \widetilde{\tau}_\varepsilon^{\, 2}= \textup{tanc}(\sqrt{d_\varepsilon}/2)^2 d_\varepsilon^2 = 4\tan(\sqrt{d_\varepsilon}/2)^2 =-4\tanh(f(\varepsilon)\sqrt{\abs{d}}/2)^2  <0.
		\end{aligned}
	\end{equation*}
	Moreover, by construction $\widetilde{V} = \widetilde{\eta}I_N + \widetilde{\tau}\beta$ with $(\widetilde{\eta},\widetilde{\tau}) = (2/\sqrt{\abs{d}})(\eta,\tau)$ and hence $\widetilde{d} = \widetilde{\eta}^{\,2} - \widetilde{\tau}^{\,2} = -4$. Thus, Proposition~\ref{prop_H_V_hat} implies that both $H_{\widetilde{V}_\varepsilon \delta_\Sigma}$ and  $H_{\widetilde{V} \delta_\Sigma}$ are self-adjoint, and that the resolvent formulas 
	\begin{equation}\label{eq_res_formulas}
		\begin{aligned}
			(H_{\widetilde{V} \delta_\Sigma} -z)^{-1}  &= (H-z)^{-1} - \Phi_z \widetilde{V} (I + \mathcal{C}_z\widetilde{V})^{-1} \Phi_z^*,\\
			(H_{\widetilde{V}_\varepsilon \delta_\Sigma} - z)^{-1}  &= (H-z)^{-1} -\Phi_z \widetilde{V}_\varepsilon (I + \mathcal{C}_z\widetilde{V}_{\varepsilon} )^{-1}\Phi_z^*
		\end{aligned}
	\end{equation}
	are valid.
	Using \eqref{eq_res_formulas} shows that the difference of the resolvents is given by
	\begin{equation}
		\begin{aligned}\label{eq_res_diff_confinment}
			-\Phi_z \widetilde{V} (I + \mathcal{C}_z\widetilde{V})^{-1} \Phi_z^* &+  \Phi_z \widetilde{V}_\varepsilon	 (I + \mathcal{C}_z\widetilde{V}_\varepsilon)^{-1}\Phi_z^*\\
			&=\Phi_z (\widetilde{V}_\varepsilon - \widetilde{V})(I+\mathcal{C}_z \widetilde{V})^{-1} \Phi_z^* \\
			&\quad+ \Phi_z \widetilde{V}_\varepsilon \big( (I + \mathcal{C}_z \widetilde{V}_\varepsilon)^{-1} - (I+\mathcal{C}_z \widetilde{V})^{-1}\big) \Phi_z^*.
		\end{aligned}
	\end{equation}
	Simple calculations give us 
	\begin{equation*}
		\begin{aligned}
			\widetilde{V}_\varepsilon &= \textup{tanc}(\sqrt{d_\varepsilon}/2)f(\varepsilon) V \\
			&= 2\frac{\tan(\sqrt{d_\varepsilon}/2)}{\sqrt{d_\varepsilon}} f(\varepsilon) V \\
			&= 2\frac{\tanh(f(\varepsilon)\sqrt{\abs{d}}/2)}{f(\varepsilon)\sqrt{\abs{d}} }f(\varepsilon) V \\
			&= \tanh(f(\varepsilon)\sqrt{\abs{d}}/2)  \widetilde{V}.
		\end{aligned}
	\end{equation*} 
	Hence,
	\begin{equation*}
		\begin{aligned}
			\lVert \widetilde{V}_\varepsilon -\widetilde{V}\rVert_{H^{1/2}(\Sigma;\C^N) \to H^{1/2}(\Sigma;\C^N)} 
			&= \bigl\|\bigl(1- \tanh(f(\varepsilon)\sqrt{\abs{d}}/2) \bigr)  \widetilde{V} \bigr\|_{H^{1/2}(\Sigma;\C^N) \to H^{1/2}(\Sigma;\C^N)} \\
			&\leq (\abs{\widetilde{\eta}} + \abs{\widetilde{\tau}}) \bigl|1- \tanh(f(\varepsilon)\sqrt{\abs{d}}/2)\bigr| \\
			&=  \frac{2(\abs{\widetilde{\eta}} + \abs{\widetilde{\tau}})}{1+e^{f(\varepsilon)\sqrt{\abs{d}}}}\\
			&\leq 2 (\abs{\widetilde{\eta}} + \abs{\widetilde{\tau}})  e^{- f(\varepsilon) \sqrt{\abs{d}}}.		
		\end{aligned}
	\end{equation*}
	In particular, if $\delta_1 > 0$ is small enough, then  
	\begin{equation*}
		\begin{aligned}
			\lVert& \mathcal{C}_z(\widetilde{V}_\varepsilon - \widetilde{V}) \rVert_{H^{1/2}(\Sigma;\C^N) \to H^{1/2}(\Sigma;\C^N)}\lVert(I + \mathcal{C}_z\widetilde{V})^{-1} \rVert_{H^{1/2}(\Sigma;\C^N) \to H^{1/2}(\Sigma;\C^N)}   \\
			&\leq C\norm{\widetilde{V} - \widetilde{V}_\varepsilon}_{H^{1/2}(\Sigma;\C^N) \to H^{1/2}(\Sigma;\C^N)} \leq \tfrac{1}{2}
		\end{aligned}
	\end{equation*}
	for all $\varepsilon \in (0,\delta_1)$.
	By the stability of bounded invertibility, see \cite[Chapter~IV, Theorem~1.16]{kato}, $I + \mathcal{C}_z \widetilde{V}_\varepsilon$ is boundedly invertible in $H^{1/2}(\Sigma;\C^N)$ and the operator norm of its inverse is bounded by 
	\begin{equation}\label{eq_in_I+C_zV_eps}
		\begin{aligned}
			\lVert(I + \mathcal{C}_z \widetilde{V}_\varepsilon)^{-1}&\rVert_{H^{1/2}(\Sigma;\C^N) \to H^{1/2}(\Sigma;\C^N)} \\
			&\leq \frac{\lVert(I + \mathcal{C}_z \widetilde{V})^{-1}\rVert_{H^{1/2}(\Sigma;\C^N) \to H^{1/2}(\Sigma;\C^N)}}{1-\frac{1}{2}}\\
			&= 2 \lVert(I + \mathcal{C}_z\widetilde{V})^{-1}\rVert_{H^{1/2}(\Sigma;\C^N) \to H^{1/2}(\Sigma;\C^N)} \leq C
		\end{aligned}
	\end{equation}
	for all $\varepsilon \in (0,\delta_1)$.
	Using  the estimates from above, \eqref{eq_res_diff_confinment}, and the mapping properties of $\Phi_z$, $\Phi_z^*$ and $\mathcal{C}_z$  yields
	\begin{equation*}
		\begin{aligned}
			\|(H_{\widetilde{V} \delta_\Sigma } &- z)^{-1}  - (H_{\widetilde{V}_\varepsilon} - z)^{-1} \|_{L^2(\R^{\theta};\C^N) \to L^2(\R^{\theta};\C^N)} \\
			&\leq \|\Phi_z (\widetilde{V}_\varepsilon	-\widetilde{V}) (I + \mathcal{C}_z\widetilde{V})^{-1}\Phi_z^*\|_{L^2(\R^{\theta};\C^N) \to L^2(\R^{\theta};\C^N)} \\
			&\quad+\|\Phi_z \widetilde{V}_\varepsilon \bigl((I+\mathcal{C}_z\widetilde{V}_\varepsilon)^{-1} - (I + \mathcal{C}_z \widetilde{V})^{-1}\bigr) \Phi_z^* \|_{L^2(\R^{\theta};\C^N) \to L^2(\R^{\theta};\C^N)} \\
			&\leq C \| \widetilde{V}_\varepsilon	-\widetilde{V} \|_{H^{1/2}(\Sigma;\C^N) \to H^{1/2}(\Sigma;\C^N)} \\
			&\quad+ \|(I+\mathcal{C}_z\widetilde{V}_\varepsilon)^{-1} - (I + \mathcal{C}_z \widetilde{V})^{-1} \|_{H^{1/2}(\Sigma;\C^N) \to H^{1/2}(\Sigma;\C^N)} \\
			&\leq C e^{-f(\varepsilon)\sqrt{\abs{d}}} \\
			&\quad+ \|(I+\mathcal{C}_z\widetilde{V})^{-1}\mathcal{C}_z (\widetilde{V}-\widetilde{V}_\varepsilon)(I + \mathcal{C}_z \widetilde{V}_\varepsilon)^{-1} \|_{H^{1/2}(\Sigma;\C^N) \to H^{1/2}(\Sigma;\C^N)} \\
			&\leq C \bigl(  e^{-f(\varepsilon)\sqrt{\abs{d}}} + \| \widetilde{V}_\varepsilon	-\widetilde{V} \|_{H^{1/2}(\Sigma;\C^N) \to H^{1/2}(\Sigma;\C^N)} \bigr) \\
			&\leq C  e^{-f(\varepsilon)\sqrt{\abs{d}}}
		\end{aligned}	
	\end{equation*}
	for all $\varepsilon \in (0,\delta_1)$, which  completes the proof.
\end{proof}

\section{A bound for (\ref{eq_res_dif_2})}\label{sec_res_dif_2}
In this section we find the estimate for the resolvent difference of  $H_{f(\varepsilon)V_\varepsilon }$ and $H_{\widetilde{V}_\varepsilon \delta_\Sigma}$ given by \eqref{eq_res_dif_2}. Here,  we proceed as follows: After introducing necessary operators, we find in Proposition~\ref{prop_conv_res} comparable resolvent formulas for  $H_{f(\varepsilon)V_\varepsilon }$ and $H_{\widetilde{V}_\varepsilon \delta_\Sigma}$. Moreover, we provide in this proposition  convergence results for the operators which appear in the resolvent formulas. Then, after two preparatory lemmas, we use these convergence estimates to prove \eqref{eq_res_dif_2} in Proposition~\ref{prop_dif_2}.

First, we recall  $\iota$ from \eqref{eq_iota} and  introduce for $\varepsilon \in (0,\varepsilon_1)$ the mappings
\begin{equation*}
	\begin{aligned}
		&\mathcal{P}_\varepsilon : \mathcal{B}^0(\Sigma) \to L^2(\Omega_\varepsilon;\C^N), &
		&\mathcal{P}_\varepsilon f (\iota(x_\Sigma,t)) := \frac{1}{\sqrt{\varepsilon}}f(t/\varepsilon)(x_\Sigma),\\
		&\mathcal{P}_\varepsilon^{-1} : L^2(\Omega_\varepsilon;\C^N) \to \mathcal{B}^0(\Sigma), &
		&\mathcal{P}_\varepsilon^{-1} u (t)(x_\Sigma) := \sqrt{\varepsilon} u(\iota(x_\Sigma,\varepsilon t)),
	\end{aligned}
\end{equation*}
which are  well-defined and bounded according to \cite[eqs. (3.2) and (3.3) with $\mathcal{P}_\varepsilon = \mathcal{I}_\varepsilon \mathcal{S}_\varepsilon$]{BHS23}. 
Next, we set $u_\varepsilon:= \chi_{\Omega_\varepsilon}/\sqrt{\varepsilon}$, where $\chi_{\Omega_\varepsilon}$ is the characteristic function for $\Omega_\varepsilon$, and we define the operators 
\begin{equation*}
	U_\varepsilon: L^2(\mathbb{R}^\theta; \mathbb{C}^N) \to L^2(\Omega_\varepsilon; \mathbb{C}^N)\quad\text{and}\quad 
	U_\varepsilon^*: L^2(\Omega_\varepsilon; \mathbb{C}^N) \to L^2(\mathbb{R}^\theta; \mathbb{C}^N)
\end{equation*}
acting on $u \in L^2(\mathbb{R}^\theta; \mathbb{C}^N)$ and $v \in L^2(\Omega_\varepsilon; \mathbb{C}^N)$ as
\begin{equation*}
	U_\varepsilon u = (u_\varepsilon u)\upharpoonright \Omega_\varepsilon \quad \text{and} \quad U_\varepsilon^* v = \begin{cases} u_\varepsilon v &\text{ in } \Omega_\varepsilon, \\ 0 & \text{ in } \mathbb{R}^\theta \setminus \Omega_\varepsilon. \end{cases}
\end{equation*}
We  provide in Proposition~\ref{prop_conv_res}~(ii) a resolvent formula for $H_{f(\varepsilon)V_\varepsilon}$, $\varepsilon \in (0,\varepsilon_1)$, in terms of the resolvent of the free Dirac operator $R_z$, $z \in \rho(H)$, and the following operators
\begin{equation} \label{eq_def_ABC_operators}
	\begin{aligned}	
		A_\varepsilon( z) :=R_ z U_\varepsilon^* \mathcal{P}_\varepsilon &:   \mathcal{B}^0(\Sigma) \to L^2(\R^\theta;\C^N), \\
		B_\varepsilon( z) := \mathcal{P}_\varepsilon^{-1} U_\varepsilon R_ z U_\varepsilon^* \mathcal{P}_\varepsilon &:  \mathcal{B}^0(\Sigma) \to  \mathcal{B}^0(\Sigma),  \\
		C_\varepsilon( z) := \mathcal{P}_\varepsilon^{-1} U_\varepsilon R_ z&: L^2(\R^\theta;\C^N) \to  \mathcal{B}^0(\Sigma).
	\end{aligned}
\end{equation}
These operators have been studied in \cite{BHS23, BHSL25, MP18} and it turns out that they converge (in suitable topologies) to their corresponding limit operators 
\begin{equation} \label{def_ABC_integral_0} 
	\begin{aligned}
		A_0(z) &: \mathcal{B}^0(\Sigma) \to L^2(\R^{\theta};\C^N), \quad A_0( z)f:= \Phi_z \int_{-1}^1 f(t) \,dt, \\
		B_0(z) &: \mathcal{B}^0(\Sigma) \to \mathcal{B}^0(\Sigma),	   \\
		B_0( z)& f(t):=\mathcal{C}_z \int_{-1}^1  f(s)ds \,+ \frac{i}{2}(\alpha \cdot \nu)\int_{-1}^1 \sign(t-s)f(s) \,ds  ,\\
		C_0(z) &: L^2(\R^{\theta};\C^N) \to \mathcal{B}^0(\Sigma), \quad C_0( z)u(t) := \Phi_{\overline{z}}^*u,
	\end{aligned}
\end{equation}
which, in turn, can be used for another resolvent representation of $H_{\widetilde{V}_\varepsilon \delta_\Sigma}$. Note that the integrals in \eqref{def_ABC_integral_0} have to be understood as Bochner integrals. We summarize all the relevant properties of the $A$, $B$, $C$ operators in the following proposition.  

\begin{prop}\label{prop_conv_res}
	Let  $V$ be as in \eqref{eq_V},  $q$ be as in \eqref{eq_q}, $V_\varepsilon$ and $\widetilde{V}_\varepsilon$ be defined by \eqref{eq_V_eps} and \eqref{eq_V_eps_tilde}, $f$ be as in \eqref{eq_conf_f}  \textup{(}with $\gamma \in (0,1/2)$\textup{)} and $z \in \rho(H)$. Then, the following is true:
	\begin{itemize}
		\item[(i)] There exists an $\varepsilon_2 > 0 $ such that  $A_\varepsilon(z)$, $B_\varepsilon(z)$, and $C_\varepsilon(z)$
		are uniformly bounded operators with respect to  $\varepsilon \in (0, \varepsilon_2)$.
		\item[(ii)] The operators $C_\varepsilon(z)$ and $C_0(z)$ also act  as uniformly bounded operators  from  $L^2(\R^{\theta};\C^N)$ to $\mathcal{B}^{1/2}(\Sigma)$ and there exists a $C >0$ such that 
		\begin{equation*}
			\begin{aligned}
				\norm{A_\varepsilon(z)-A_0(z)}_{0 \to L^2(\R^{\theta};\C^N)} &\leq C \varepsilon^{\gamma}, \\
				\norm{B_\varepsilon(z)-B_0(z)}_{1/2 \to 0} &\leq C \varepsilon^{\gamma},\\
				\norm{C_\varepsilon(z)-C_0(z)}_{L^2(\R^{\theta};\C^N) \to 0} &\leq C \varepsilon^{\gamma}
			\end{aligned}
		\end{equation*}
		for all $\varepsilon \in (0,\varepsilon_2)$ with $\varepsilon_2>0$ from \textup{(i)}.
		\item[(iii)] If  $-1 \in \rho(B_\varepsilon(z)f(\varepsilon) Vq)$, then $z \in \rho(H_{f(\varepsilon)V_\varepsilon})$ and 
		\begin{equation*}
			(H_{f(\varepsilon) V_\varepsilon}-z)^{-1} = R_z - A_\varepsilon(z)f(\varepsilon)Vq(I+B_\varepsilon(z)f(\varepsilon)Vq)^{-1} C_\varepsilon(z).
		\end{equation*}
		\item[(iv)]  If $z \in \C \setminus \R$ and \eqref{eq_conf_V} is fulfilled, then   $I+B_0(z)f(\varepsilon)Vq$ is boundedly invertible in $\mathcal{B}^r(\Sigma)$, $ r \in [0,1/2]$, and the resolvent formula 
		\begin{equation*}
			(H_{\widetilde{V}_\varepsilon \delta_\Sigma} -z)^{-1} = R_z - A_0(z) f(\varepsilon) Vq (I + B_0(z)f(\varepsilon)Vq)^{-1} C_0(z)
		\end{equation*}
		is valid.
	\end{itemize} 	
\end{prop}
\begin{proof}
	See  \cite[Propositions~3.7,~3.8 and 3.10]{BHS23} for (i) and (ii). Item (iii) is given by  \cite[Proposition~3.2]{BHS23} and (iv) follows from combining \cite[the text above Section~4.1 and Propositions~4.5 and 4.9]{BHSL25}, where $V$ in \cite{BHS23,BHSL25} must be substituted by $f(\varepsilon)V$ from the current paper.
\end{proof}

Before we prove \eqref{eq_res_dif_2} in Proposition~\ref{prop_dif_2}, we provide two preparatory lemmas.

\begin{lem}\label{lem_confinement_5}
	Let  $q$ be as in \eqref{eq_q}, $V = \eta I_N + \tau \beta$, $\eta,\tau \in \R$, satisfy \eqref{eq_conf_V}, $f$ be as in \eqref{eq_conf_f}  and $z \in \C\setminus\R$. Then, there exits a $\delta_2 >0$ and a $C>0$ such that  
	\begin{equation*}
		\bigl\|f(\varepsilon)Vq(I+B_0(z)f(\varepsilon)V q)^{-1} C_0(z)\bigr\|_{L^2(\R^{\theta};\C^N) \to 1/2} \leq C f(\varepsilon)^{1/2}
	\end{equation*}
	for all $\varepsilon \in (0,\delta_2)$.
\end{lem}
\begin{proof}
	We define $\delta_2 := \min \{\delta_1,\varepsilon_2\}  >0$ with $\delta_1$ and $\varepsilon_2$ as in Propositions~\ref{prop_dif_1} and ~\ref{prop_conv_res}, respectively, and assume throughout this proof $\varepsilon \in (0,\delta_2)$. By  Proposition~\ref{prop_conv_res}~(iv),  the operator $I+B_0(z)f(\varepsilon)V q$ is boundedly invertible in $\mathcal{B}^{1/2}(\Sigma)$ and from \cite[Lemmas~4.2~(ii) and 4.3~(iii) (with $V$ subsituted by $f(\varepsilon)V$)]{BHS23} we get
	\begin{equation*}
		\begin{aligned}
			&\bigl(f(\varepsilon)Vq(I+B_0(z)f(\varepsilon)V q)^{-1} C_0(z)u \bigr)(t) \\
			&\quad = f(\varepsilon)Vq(t)\cos((\alpha \cdot \nu) f(\varepsilon) V/2)^{-1} \exp(-i(\alpha \cdot \nu) f(\varepsilon) VQ(t))(I+ \mathcal{C}_z \widetilde{V}_\varepsilon)^{-1}\Phi_{\overline{ z}}^*u
		\end{aligned}
	\end{equation*}
	with $\widetilde{V}_\varepsilon$ from \eqref{eq_V_eps_tilde} and $Q(t) := -\tfrac{1}{2} + \int_{-1}^t q(t) \,dt$ for $t \in (-1,1)$. Here, $\cos$ and $\exp$ for matrices are defined via their power series.
	We know from the text below \eqref{eq_Phi_z} that $\Phi_{\overline{ z}}^*$ acts as a bounded operator from $L^2(\R^\theta;\C^N)$ to $H^{1/2}(\Sigma;\C^N)$. Moreover, from the proof of Proposition~\ref{prop_dif_1}, see \eqref{eq_in_I+C_zV_eps}, we also know that $\|(I + \mathcal{C}_z \widetilde{V}_\varepsilon)^{-1}\|_{H^{1/2}(\Sigma;\C^N) \to H^{1/2}(\Sigma;\C^N)}$ is uniformly bounded with respect to $\varepsilon \in (0,\delta_2)$. Hence,  we obtain for $u \in L^2(\R^{\theta};\C^N)$
	\begin{equation}\label{eq_lem_conf_1}
		\begin{aligned}
			\bigl\|f(\varepsilon)&Vq(I+B_0(z)f(\varepsilon)V q)^{-1} C_0(z)u\bigr\|_{ 1/2}^2 \\
			&=\int_{-1}^1\bigl\|\bigl(f(\varepsilon)Vq(I+B_0(z)f(\varepsilon)V q)^{-1} C_0(z)u \bigr)(t) \bigr\|_{ H^{1/2}(\Sigma;\C^N)}^2 \, dt \\
			&=\int_{-1}^1 \bigl\| f(\varepsilon)Vq(t)\cos\big((\alpha \cdot \nu) f(\varepsilon) V/2 \big)^{-1} \\
			&\hspace{18 pt}\cdot \exp(-i(\alpha \cdot \nu) f(\varepsilon) VQ(t))(I+ \mathcal{C}_z \widetilde{V}_\varepsilon)^{-1}\Phi_{\overline{ z}}^* u \bigr\|_{H^{1/2}(\Sigma;\C^N)}^2 \,dt\\
			& \leq  C  \int_{-1}^1  f(\varepsilon)^2q(t)^2 \bigl\| \cos\big((\alpha \cdot \nu) f(\varepsilon) V/2 \big)^{-1} \\
			&\hspace{18 pt} \cdot \exp(-i(\alpha \cdot \nu) f(\varepsilon) VQ(t))\bigr\|_{H^{1/2}(\Sigma;\C^N) \to H^{1/2}(\Sigma;\C^N)}^2 \, dt \lVert  u \rVert_{L^2(\R^\theta,\C^N)}^2.
		\end{aligned}
	\end{equation}
	Next, we fix $t \in (-1,1)$. The identity $((\alpha \cdot \nu) V)^2 = (\eta^2-\tau^2) I_N = d I_N$, $ d = - |d|$, $f(\varepsilon) >0$, and the power series representations of $\cos$, $\cosh$ and $\sinh$ lead to
	\begin{equation}\label{eq_lem_conf_2}
		\begin{aligned}
			\cos((\alpha \cdot \nu) &f(\varepsilon) V/2 )^{-1} \exp(-i(\alpha \cdot \nu) f(\varepsilon) VQ(t)) \\
			&=  \frac{\cosh(f(\varepsilon) \sqrt{\abs{d}} Q(t)) + (\alpha \cdot \nu) \frac{V}{\sqrt{d}} \sinh(f(\varepsilon)\sqrt{\abs{d}}Q(t))}{\cosh(f(\varepsilon)\sqrt{\abs{d}}/2)}.
		\end{aligned}
	\end{equation}
	Since $\Sigma$ satisfies the assumptions from Section~\ref{sec_not}~\eqref{it_def_Sigma}, it follows from \cite[eq.~(2.3)]{BHS23} that $\alpha \cdot \nu $ acts as a bounded operator in $H^{1/2}(\Sigma;\C^N)$. Thus, by \eqref{eq_lem_conf_2}  we get
	\begin{equation}\label{eq_lem_conf_3}
		\begin{aligned}
			&\bigl\| \cos((\alpha  \cdot \nu) f(\varepsilon) V/2 )^{-1} \exp(-i(\alpha \cdot \nu) f(\varepsilon) VQ(t))\bigr\|_{H^{1/2}(\Sigma;\C^N) \to H^{1/2}(\Sigma;\C^N)} \\
			&\hspace{100 pt} \leq  \frac{\cosh(f(\varepsilon)\sqrt{\abs{d}}Q(t)) +C\abs{\sinh(f(\varepsilon)\sqrt{\abs{d}} Q(t))}}{\cosh(f(\varepsilon)\sqrt{\abs{d}}/2)}\\
			&\hspace{100 pt} \leq C\frac{e^{f(\varepsilon)\sqrt{\abs{d}} \abs{Q(t)}}}{\cosh(f(\varepsilon)\sqrt{\abs{d}}/2)},
		\end{aligned}
	\end{equation}
	where $C>0$ does not depend on $t$.
	Hence, plugging \eqref{eq_lem_conf_3} into \eqref{eq_lem_conf_1}, and using $q \geq 0$ a.e. as well as $Q(1) =  Q(-1) = 1/2$, see \eqref{eq_q}, yields 
	\begin{equation*}
		\begin{aligned}
			\bigl\|f(\varepsilon)&Vq(t)(I+B_0(z)f(\varepsilon)V q)^{-1} C_0(z)u\bigr\|_{ 1/2}^2  \\ &\leq \frac{C}{\cosh(f(\varepsilon)\sqrt{\abs{d}}/2)^2} \int_{-1}^1f(\varepsilon)^2 q(t)^2e^{2f(\varepsilon)\sqrt{\abs{d}} \abs{Q(t)}}\,dt  \lVert  u \rVert_{L^2(\R^\theta,\C^N)}^2\\
			&\leq \frac{C  f(\varepsilon) \norm{q}_{L^\infty((-1,1))}}{\sqrt{\abs{d}}\cosh(f(\varepsilon)\sqrt{\abs{d}}/2)^2}\int_{-1}^1 e^{2f(\varepsilon)\sqrt{\abs{d}} \abs{Q(t)}} f(\varepsilon) \sqrt{\abs{d}} q(t) \,dt  \lVert  u \rVert_{L^2(\R^\theta,\C^N)}^2\\
			&\leq \frac{C   f(\varepsilon)}{\cosh(f(\varepsilon)\abs{d}/2)^2}  \int_{-f(\varepsilon)\sqrt{\abs{d}}/2}^{f(\varepsilon)\sqrt{\abs{d}}/2}  e^{2\abs{s}}\,ds  \lVert  u \rVert_{L^2(\R^\theta,\C^N)}^2\\
			&\leq \frac{C f(\varepsilon)}{\cosh(f(\varepsilon)\sqrt{\abs{d}}/2)^2} \int_{0}^{f(\varepsilon)\sqrt{\abs{d}}/2}  e^{2s} \,ds  \lVert  u \rVert_{L^2(\R^\theta,\C^N)}^2\\
			& \leq  C f(\varepsilon)\frac{e^{f(\varepsilon)\sqrt{\abs{d}}}-1}{\cosh(f(\varepsilon)\sqrt{\abs{d}}/2)^2}  \lVert  u \rVert_{L^2(\R^\theta,\C^N)}^2\\
			&\leq C f(\varepsilon)  \lVert  u \rVert_{L^2(\R^\theta,\C^N)}^2.
		\end{aligned}
	\end{equation*} 
\end{proof}

In the next lemma we use the auxiliary operator  $M_{\varepsilon} : \mathcal{B}^0(\Sigma) \to \mathcal{B}^0(\Sigma)$, $\varepsilon \in (0,\varepsilon_1)$, defined by 
\begin{equation*} 
	\begin{split}
		M_\varepsilon f(t) = \det \left(I- t \varepsilon W \right)f(t)\qquad \text{ for a.e. } t \in(-1,1), 
	\end{split}
\end{equation*}
where $W$ is the Weingarten map introduced in Section~\ref{sec_not}~\eqref{it_def_Sigma}. According to \cite[Lemma~3.6]{BHS23}  this operator is bounded, invertible, 
\begin{equation}\label{eq_M_eps_norm_est}
	\norm{M_{\varepsilon}}_{0\to0} \leq (1+\varepsilon C)\quad \text{and} \quad\norm{M_{\varepsilon} - I}_{0 \to 0} \leq \varepsilon C.
\end{equation}
Moreover, from the proof of \cite[Proposition~3.8 (with $\mathcal{P}_\varepsilon = \mathcal{I}_\varepsilon \mathcal{S}_\varepsilon$)]{BHS23} we obtain that $M_\varepsilon$ is connected to $\mathcal{P}_\varepsilon$ via the formula 
\begin{equation}\label{eq_P_eps_adjoint}
	\mathcal{P}_\varepsilon^* = M_\varepsilon \mathcal{P}_\varepsilon^{-1}.
\end{equation}
Having introduced and discussed the operator $M_\varepsilon$ we turn to the second preparatory lemma for Proposition~\ref{prop_dif_2}.

\begin{lem}\label{lem_confinement_2}
	Let $m = 0$, $q$ be as in \eqref{eq_q},   $V = \eta I_N + \tau \beta$, $\eta,\tau \in \R$ satisfy \eqref{eq_conf_V}, $f$ be as in \eqref{eq_conf_f} and  $z \in i \R$. Then, there  exists a $\delta_3 > 0$ and a $C>0$ such that $(I+B_\varepsilon(z)f(\varepsilon)V q)^{-1}$ is boundedly invertible in $\mathcal{B}^0(\Sigma)$ and
	\begin{equation*}
		\norm{f(\varepsilon)Vq(I+B_\varepsilon(z)f(\varepsilon)V q)^{-1}}_{0 \to 0} \leq C f(\varepsilon)
	\end{equation*}
	for all $\varepsilon \in (0,\delta_3)$.
\end{lem}
\begin{proof}
	We start by defining $\widetilde{B}_\varepsilon(z):= B_\varepsilon(z)M_\varepsilon^{-1}$. Note that according to \eqref{eq_def_ABC_operators} and \eqref{eq_P_eps_adjoint} the equality
	\begin{equation*}
		(\widetilde{B}_\varepsilon(z))^*=  (\mathcal{P}_\varepsilon^{-1} U_\varepsilon R_ z U_\varepsilon^* \mathcal{P}_\varepsilon M_\varepsilon ^{-1})^* = M_\varepsilon ^{-1} M_\varepsilon \mathcal{P}_\varepsilon^{-1} U_\varepsilon R_ {\overline{z}} U_\varepsilon^* \mathcal{P}_\varepsilon M_\varepsilon^{-1} =  \widetilde{B}_\varepsilon(\overline{z})
	\end{equation*}
	holds in $\mathcal{B}^0(\Sigma)$. We also introduce
	\begin{equation*}
		D = \sqrt{q} \textup{diag}(\sqrt{|\eta + \tau|} I_{N/2}, \sqrt{|\eta - \tau|} I_{N/2})
	\end{equation*} 
	and
	\begin{equation*}
		E_\varepsilon(z) :=  \beta+ f(\varepsilon)\sign(\tau) D  \widetilde{B}_\varepsilon(z)  D: \mathcal{B}^0(\Sigma) \to \mathcal{B}^0(\Sigma).
	\end{equation*}  
	Next, we estimate the norm of the inverse of $E_\varepsilon(z)$. By $\beta^2 = I_N$ we obtain
	\begin{equation*}
		\begin{aligned}
			E_\varepsilon(z) (E_\varepsilon(z))^* &=  I  + f(\varepsilon)\sign(\tau) D(  \widetilde{B}_\varepsilon(z)\beta +\beta(\widetilde{B}_\varepsilon(z))^* )D 
			\\
			&\quad  +  f(\varepsilon)^2 D  \widetilde{B}_\varepsilon(z)  D (D \widetilde{B}_\varepsilon(z)D)^* .
		\end{aligned}	
	\end{equation*}
	We claim that the operator 
	\begin{equation*}
		\widetilde{B}_\varepsilon(z)\beta + \beta(\widetilde{B}_\varepsilon(z))^*  =  \widetilde{B}_\varepsilon(z) \beta + \beta \widetilde{B}_\varepsilon(\overline{z})  = \mathcal{P}_\varepsilon^{-1} U_\varepsilon ( R_z \beta + \beta  R_{\overline{z}}) U_\varepsilon^* \mathcal{P}_\varepsilon  M_\varepsilon^{-1} 
	\end{equation*}
	vanishes under the current assumptions. In fact, since $z \in i \R$, we have ${\overline{z} = -z}$, and as $m =0$,  the anticommutation rules from Section~\ref{sec_not}~\eqref{it_Dirac_matrices} yield $ H \beta= -\beta H$. Hence,
	\begin{equation*}
		R_z \beta + \beta R_{\overline{z}}  = (H-z)^{-1}  \beta + \beta (H-\overline{z})^{-1}  = \beta(-H-z)^{-1}  + \beta(H+z)^{-1}   = 0.
	\end{equation*}	
	This implies
	\begin{equation*}
		E_\varepsilon(z) (E_\varepsilon(z))^* = I + f(\varepsilon)^2 D  \widetilde{B}_\varepsilon(z)  D (D \widetilde{B}_\varepsilon(z)D)^*.
	\end{equation*}
	Analogously, one obtains
	\begin{equation*}
		(E_\varepsilon(z))^* E_\varepsilon(z) =I + f(\varepsilon)^2(D  \widetilde{B}_\varepsilon(z)  D)^* D \widetilde{B}_\varepsilon(z)D.
	\end{equation*}
	Thus, the self-adjoint operators $(E_\varepsilon(z))^* E_\varepsilon(z)$ and $E_\varepsilon(z) (E_\varepsilon(z))^*$ are bounded from below by one. In particular, $(E_\varepsilon(z))^* E_\varepsilon(z)$ and $E_\varepsilon(z) (E_\varepsilon(z))^*$ are boundedly invertible in $\mathcal{B}^0(\Sigma)$. Hence, the operators $ \bigl((E_\varepsilon(z))^* E_\varepsilon(z)\bigr)^{-1}(E_\varepsilon(z))^*$  and $(E_\varepsilon(z))^*\bigl(E_\varepsilon(z) (E_\varepsilon(z))^*\bigr)^{-1}$ are a bounded left and a bounded right inverse of $E_\varepsilon(z)$, respectively. Therefore, $E_\varepsilon(z)$ is boundedly invertible in $\mathcal{B}^{0}(\Sigma)$. Moreover, the   norm estimate 
	\begin{equation*}
		\norm{E_\varepsilon(z) f}_0 = \sqrt{((E_\varepsilon(z))^*E_\varepsilon(z) f,f)_{\mathcal{B}^0(\Sigma)}} \geq  \norm{f}_{0}, \quad f \in \mathcal{B}^0(\Sigma),
	\end{equation*}
	gives us
	\begin{equation}\label{eq_E_eps_inv_est}
		\norm{(E_\varepsilon(z))^{-1}}_{0 \to 0} \leq 1.
	\end{equation}
	
	Next, we choose $\delta_3   \in (0,\varepsilon_2]$, with $\varepsilon_2$ from Proposition~\ref{prop_conv_res}, sufficiently small such that 
	\begin{equation*}
		\|D (\widetilde{B}_\varepsilon(z) - B_\varepsilon(z)) D\|_{0 \to 0}  = \|D B_\varepsilon(z) ( M_\varepsilon^{-1} - I) D\|_{0 \to 0}\leq \frac{1}{2f(\varepsilon)}
	\end{equation*} 
	for all $\varepsilon \in (0,\delta_3)$ which is possible according to \eqref{eq_M_eps_norm_est}, Proposition~\ref{prop_conv_res} and the choice of $f$ in \eqref{eq_conf_f}. This estimate, 
	\begin{equation*}
		E_\varepsilon(z) - (\beta + f(\varepsilon)\sign(\tau)D B_\varepsilon(z)D) = f(\varepsilon)\sign(\tau)D (\widetilde{B}_\varepsilon(z) - B_\varepsilon(z)) D,
	\end{equation*}
	\eqref{eq_E_eps_inv_est} and the stability of bounded invertibility, see \cite[Chapter~IV, Theorem~1.16]{kato}, show that  the operator $ \beta + f(\varepsilon)\sign(\tau)  D B_\varepsilon(z)D$ is also boundedly invertible in $\mathcal{B}^0(\Sigma)$ and 
	\begin{equation*}
		\begin{aligned}
			\bigl\|(\beta + f(\varepsilon)&\sign(\tau)  D B_\varepsilon(z)D)^{-1}\bigr\|_{0 \to 0} \\
			&\leq  \frac{ \bigl\| (E_\varepsilon(z))^{-1} \bigr\|_{0 \to 0}}{1-\|f(\varepsilon)\sign(\tau) D ( B_\varepsilon(z)  -\widetilde{B}_\varepsilon(z) ) D \|_{0\to 0} \bigl\|(E_\varepsilon(z))^{-1}\bigr\|_{0 \to 0}}\\
			&\leq \frac{1}{1-\frac{f(\varepsilon)}{2f(\varepsilon)}} = 2  
		\end{aligned}
	\end{equation*}
	for all $ \varepsilon \in (0,\delta_3)$. Hence, the same is true for 
	\begin{equation*}
		I + f(\varepsilon) \sign(\tau) \beta D B_\varepsilon(z)D =  \beta (\beta + f(\varepsilon)\sign( \tau) D B_\varepsilon(z)D).
	\end{equation*} 
	Applying \cite[Proposition~2.1.8~(a)]{P94} and the identity $Vq = \sign(\tau) D \beta D $ yields that the operator $I + B_\varepsilon (z)f(\varepsilon)Vq$ is also boundedly invertible in $\mathcal{B}^0(\Sigma)$. Moreover, it is also easy to check that the identity 
	\begin{equation*}
		Vq(I + B_\varepsilon(z)f(\varepsilon)Vq)^{-1} = D(I + f(\varepsilon) \sign(\tau) \beta D B_\varepsilon(z)D)^{-1}\sign(\tau)\beta D
	\end{equation*}
	is valid. This shows
	\begin{equation*}
		\begin{aligned}
			\bigl\|f(\varepsilon)Vq(I + B_\varepsilon(z)f(\varepsilon)Vq)^{-1}\bigr\|_{0 \to 0} 
			&= f(\varepsilon)\bigl\|Vq(I + B_\varepsilon(z)f(\varepsilon)Vq)^{-1}\bigr\|_{0 \to 0} \\
			& \leq C f(\varepsilon)	\bigl\|(I + f(\varepsilon) \sign(\tau) \beta D B_\varepsilon(z)D)^{-1}\bigr\|_{0 \to 0}\\
			&\leq  C f(\varepsilon)
		\end{aligned}
	\end{equation*}
	for all $\varepsilon \in (0,\delta_3)$, which completes the proof.
\end{proof}

Finally, with the help of Proposition~\ref{prop_conv_res} and Lemmas~\ref{lem_confinement_5} and ~\ref{lem_confinement_2} we are able to prove \eqref{eq_res_dif_2} in the upcoming proposition.

\begin{prop}\label{prop_dif_2} 
	Let  $q$ be as in \eqref{eq_q}, $V  = \eta I_N + \tau \beta$, $\eta,\tau \in \R$ satisfy \eqref{eq_conf_V}, $V_\varepsilon$ be defined by \eqref{eq_V_eps}, $f$ be as in \eqref{eq_conf_f} \textup{(}with $\gamma \in (0,1/2)$\textup{)} and  $z \in \C\setminus\R$. Moreover, let $\widetilde{V} = (2/\sqrt{|d|})V$ and $\widetilde{V}_\varepsilon =  \textup{tanc}(\sqrt{d_\varepsilon}/2)f(\varepsilon) V$, where $d = \eta^2 -\tau^2$ and $d_\varepsilon = f(\varepsilon)^2 d$. Then, there exists a  $\delta_4>0$ and a $C>0$ such that
	\begin{equation*}
		\lVert(H_{\widetilde{V}_\varepsilon \delta_\Sigma} -z)^{-1} -(H_{f(\varepsilon)V_\varepsilon} -z)^{-1} \rVert_{L^2(\R^{\theta};\C^N) \to L^2(\R^{\theta};\C^N)} \leq C f(\varepsilon)^{3/2} \varepsilon^{\gamma} 
	\end{equation*}
	for all $ \varepsilon \in  (0,\delta_4)$.
\end{prop}
\begin{proof}
	First, let us mention that we can assume w.l.o.g. $m=0$ and $z \in i \R$ since bounded self-adjoint perturbations do not influence the norm resolvent convergence; see for instance \cite[Chapter~IV, Remark~2.16 and Theorems~2.17 and 2.20]{kato}.
	
	We set $\delta_4 := \min\{\delta_2,\delta_3\} > 0$ with $\delta_2$ and $\delta_3$ from Lemmas~\ref{lem_confinement_5} and \ref{lem_confinement_2}, respectively, and assume throughout this proof $\varepsilon \in (0,\delta_4)$. Then, Lemma~\ref{lem_confinement_2} shows  $-1 \in \rho(B_\varepsilon(z)Vq)$. Moreover, applying the resolvent formulas from Proposition~\ref{prop_conv_res} (iii) and (iv) lets us write
	\begin{equation}\label{eq_res_diff_R}
		\begin{aligned}
			(H_{\widetilde{V}_\varepsilon \delta_\Sigma } -z)^{-1} & - (H_{f(\varepsilon)V_\varepsilon} -z)^{-1} \\
			&\quad  = -A_0(z) f(\varepsilon)Vq (I+B_0(z)f(\varepsilon)V q)^{-1} C_0(z) \\
			&\qquad  + A_\varepsilon(z) f(\varepsilon)Vq (I+B_\varepsilon(z)f(\varepsilon)V q)^{-1} C_\varepsilon(z)\\
			&\quad = \mathcal{R}_1(z) + \mathcal{R}_2(z) + \mathcal{R}_3(z)
		\end{aligned}
	\end{equation}
	with 
	\begin{equation*}
		\begin{aligned}
			\mathcal{R}_1(z) &:=  (A_\varepsilon(z)-A_0(z)) f(\varepsilon)Vq (I+B_0(z)f(\varepsilon)V q)^{-1} C_0(z), \\
			\mathcal{R}_2(z) &:=  A_\varepsilon(z) f(\varepsilon)Vq \bigl((I+B_\varepsilon(z)f(\varepsilon)V q)^{-1}-(I+B_0(z)f(\varepsilon)V q)^{-1}\bigr)  C_0(z),\\
			\mathcal{R}_3(z) &:= A_\varepsilon(z) f(\varepsilon)Vq (I+B_\varepsilon(z)f(\varepsilon)V q)^{-1} (C_\varepsilon(z) - C_0(z)).
		\end{aligned}
	\end{equation*}
	Next, we estimate $\mathcal{R}_1(z)$, $\mathcal{R}_2(z)$ and $\mathcal{R}_3(z)$.
	Applying Proposition~\ref{prop_conv_res}~(i) and (ii), and Lemma~\ref{lem_confinement_5} yields
	\begin{equation*}
		\begin{aligned}
			\lVert&\mathcal{R}_1(z)\rVert_{L^2(\R^\theta;\C^N) \to L^2(\R^\theta;\C^N)}\\
			& \leq \norm{A_\varepsilon(z)-A_0(z)}_{0\to L^2(\R^{\theta};\C^N)} \norm{f(\varepsilon)V q (I+B_0(z)f(\varepsilon)V q)^{-1} C_0(z) }_{L^2(\R^\theta;\C^N) \to 0} \\
			&\leq \norm{A_\varepsilon(z)-A_0(z)}_{0 \to L^2(\R^{\theta};\C^N)} \| f(\varepsilon)V q (I+B_0(z)f(\varepsilon)V q)^{-1} C_0(z) \|_{L^2(\R^\theta;\C^N) \to 1/2}\\
			&\leq C \varepsilon^{\gamma} f(\varepsilon)^{1/2}.
		\end{aligned}
	\end{equation*}
	Using  Lemma~\ref{lem_confinement_2} additionally gives us
	\begin{equation*}
		\begin{aligned}
			&\lVert\mathcal{R}_2(z)\rVert_{L^2(\R^\theta;\C^N) \to L^2(\R^\theta;\C^N)}
			\\
			& =  \lVert A_\varepsilon(z) f(\varepsilon)V q(I+B_\varepsilon(z)f(\varepsilon)V q)^{-1} \\
			&\qquad \cdot(B_0(z) - B_\varepsilon(z)) f(\varepsilon)V q (I+B_0(z)f(\varepsilon)V q)^{-1} C_0(z) \rVert_{L^2(\R^{\theta};\C^N) \to L^2(\R^{\theta};\C^N)} \\
			& \leq \norm{A_\varepsilon(z)}_{0 \to L^2(\R^\theta;\C^N)} \norm{f(\varepsilon)V q(I+B_\varepsilon(z)f(\varepsilon)V q)^{-1}}_{0 \to 0}\\
			&\qquad\cdot \norm{B_\varepsilon(z) - B_0(z)}_{1/2 \to 0} \| f(\varepsilon)V q (I\!+\!B_0(z)f(\varepsilon)V q)^{-1} C_0(z) \|_{L^2(\R^\theta;\C^N) \to 1/2}\\
			&\leq C \varepsilon^{\gamma} f(\varepsilon)^{3/2}.
		\end{aligned}
	\end{equation*}
	Similarly, we get
	\begin{equation*}
		\begin{aligned}
			\lVert\mathcal{R}_3(z)\rVert_{L^2(\R^\theta;\C^N) \to L^2(\R^\theta;\C^N)}
			&\leq \norm{A_\varepsilon(z)}_{0\to L^2(\R^{\theta};\C^N)} \norm{f(\varepsilon)V q (I+B_\varepsilon(z)f(\varepsilon)V q)^{-1}}_{0 \to 0} \\
			& \quad\cdot\norm{C_\varepsilon(z) - C_0(z) }_{L^2(\R^\theta;\C^N) \to 0} \\
			&\leq C \varepsilon^{\gamma} f(\varepsilon).
		\end{aligned}
	\end{equation*}
	Combining these estimates with \eqref{eq_res_diff_R} yields the assertion.
\end{proof}

\section*{Acknowledgment}
This research was funded in whole by the Austrian Science Fund (FWF) 10.55776/P 33568-N. For the purpose of open access, the author has applied a CC BY public copyright licence to any Author Accepted Manuscript version arising from this submission.

\bibliographystyle{abbrv}

\end{document}